\documentclass{amsart}%
\usepackage{amsmath}%
\usepackage{amsfonts}%
\usepackage{amssymb}%
\usepackage{graphicx}
\usepackage{pstricks}
\usepackage{pstricks-add}
\usepackage{enumerate}
\newtheorem{theorem}{Theorem}
\theoremstyle{plain}

\newtheorem{algorithm}{Algorithm}

\newtheorem{conjecture}{Conjecture}

\newtheorem{definition}{Definition}

\newtheorem{lemma}{Lemma}

\newtheorem{remark}{Remark}

\begin{document}

\title[Solving an initial value problem with a given tolerance]{The method of solving a scalar initial value problem with a required tolerance}
\author{Alexander V. Lozovskiy}
\address[A. Lozovskiy]
{334, Laboratoire de Mecanique et Genie Civil de Montpellier, Universite Montpellier 2\newline%
\indent 34095 Montpellier, France}%
\email[A. Lozovskiy]{alexander.lozovskiy@univ-montp2.fr}%
\date{May 12, 2011}


\begin{abstract}
A new numerical method for solving a scalar ordinary differential equation with a given initial condition is introduced. The method is using a numerical integration procedure for an equivalent integral equation and is called in this paper an integrating method. Bound to specific constraints, the method returns an approximate solution assuredly within a given tolerance provided by a user. This makes it different from a large variety of single- and multi-step methods for solving initial value problems that provide results up to some undefined error in the form $O(h^{k})$, where $h$ is a step size and $k$ is concerned with the method's accuracy. Advantages and disadvantages of the method are presented. Some improvements in order to avoid the latter are also made. Numerical experiments support these theoretical results.
\end{abstract}

\maketitle
\section{An idea of the method}
Consider a following differential equation supplied with an initial condition:
\begin{equation}\label{E:equation_main}
\begin{cases}
\frac{dy}{dx} = f(y)\cdot g(x), \text{ }0\leqslant x \leqslant b,\\
y(0) = y_{0}.
\end{cases}
\end{equation}
The purpose is to find $y(b)$, provided the solution can be extended from the initial condition at $x = 0$ to $x=b$. Classical single or multi-step methods, e.g. Runge-Kutta or Adams schemes, would discretize the interval $[0, b]$ into partition $\{0, x_{1}, x_{2}, ..., x_{N-1}, x_{N}\}$ with $x_{N} = b$ and compute iterative relations that would lead to some approximation of $y(b)$. The answer would be in the form $y(b) = y_{x_{N}} + O(h^{p})$, where $y_{x_{i}}$ denotes the approximate solution at point $x_{i}$ of the partition, $h$ is a characteristic mesh size and $p$ is the order of the method depending on the approximating scheme. The error $O(h^{p})$, in general, cannot be estimated. Some bounds may be found only in certain cases. These methods are mostly used as trustworthy, assumed $p$ is large enough to drive the whole error $O(h^{p})$ to zero. Negative effects may happen if, for example, unstability takes place.  

Now, assume functions $f$ and $g$ from \eqref{E:equation_main} satisfy the following conditions.

\begin{enumerate}[(A)]
\item $\int_{0}^{b}g(x)dx$ can be evaluated exactly.
\item Let $\tau(z) = \int_{0}^{z}g(x)dx$. Then $\tau^{'}(z) = g(z) > 0$ for any $z > 0$.
\item $f(y_{0}) > 0$ and $f^{'}(y) > 0$ for any $y \geqslant y_{0}$.
\item $\left(\frac{1}{f(y)}\right)^{''} > 0$ for any $y \geqslant y_{0}$.
\end{enumerate}
Let us call those the integrating conditions. From \eqref{E:equation_main}, using the separation of variables, it is easy to obtain an equivalent relation
\[
\int_{y_{0}}^{y(b)}\frac{dy}{f(y)} = \int_{0}^{b}g(x)dx.
\]
The term on the right-hand side is non other than $\tau(b)$. By assumption (A) of the integrating conditions, this term is given precisely, i.e. without an error. From the numerical point of view, this and assumption (B) of the integrating conditions allow to replace $\tau(b)$ with $b$ without loss of generality to reduce problem \eqref{E:equation_main} to
\begin{equation}\label{E:equation}
\begin{cases}
\frac{dy}{dx} = f(y), \text{ }0\leqslant x \leqslant b,\\
y(0) = y_{0},
\end{cases}
\end{equation}
and
\begin{equation}\label{E:integrating_equation}
\int_{y_{0}}^{y(b)}\frac{dy}{f(y)} = b.
\end{equation}
Denote
\[
p(y) := \frac{1}{f(y)}.
\]
So \eqref{E:integrating_equation} turns into
\begin{equation}\label{E:integrating_equation_final}
\int_{y_{0}}^{y(b)}p(y)dy = b.
\end{equation}
From now on, we will be focusing on integral equation \eqref{E:integrating_equation_final}. This is the core idea of the integrating method that will be presented below. The qualitative graph of $p(y)$ is shown on picture \ref{fig:func_py}. Such behavior is due to assumptions (C) and (D) of the integrating conditions.

\begin{figure}
\begin{pspicture}(4,1)(8,6.3)
   \pscurve[linewidth=0.5pt, linecolor=black]{->}(4,4.2)(5.8,2.8)(8,2.2)(10,2.1)
   \psline[linewidth=0.5pt]{<->}(2.5,5.9)(2.5,2)(11,2)
   \psline[linewidth=0.5pt, linecolor=blue, linestyle=dashed](4,4.2)(4,2)
   \psline[linewidth=0.5pt, linecolor=blue, linestyle=dashed](4,4.2)(2.5,4.2)
   \uput[0](10.6,1.7){$y$}
   \uput[0](3.6,1.7){$y_{0}$}
   \uput[0](1.5,4.2){$p(y_{0})$}
%   \uput[0](5.4,1.2){Function $p(y)$}
%   \uput[0](5.5,0.5){Picture 1}
 \end{pspicture}
\caption{Function $p(y)$}
\label{fig:func_py}
\end{figure}
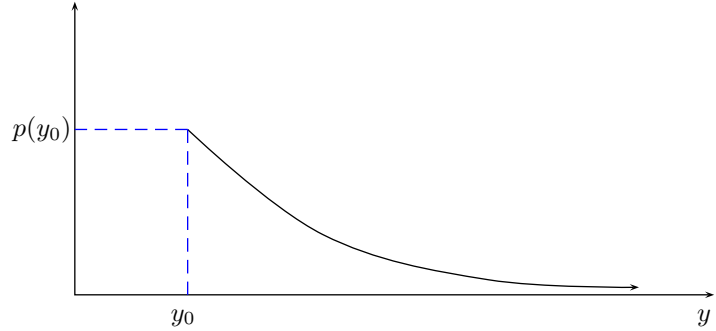
\begin{remark}
The condition that $\int_{y_{0}}^{+\infty}p(y)dy = c$ is equivalent to the condition that the solution of \eqref{E:equation} can only be extended to point $x = c$. Therefore, it is necessary to impose condition $b < c$, in order to deal with a solvable problem.  
\end{remark}

In order to solve \eqref{E:integrating_equation_final} for $y(b)$, consider the lower rectangular and the trapezoidal methods with constant step $h$ for the integral in \eqref{E:integrating_equation_final}, picture \ref{fig:int_rec_trap}.

\begin{figure}
\begin{pspicture}(4,1)(8,6.3)
   \pscurve[linewidth=0.25pt, linecolor=black]{->}(2,5)(2.5,4)(3,3.5)(3.5,3.2)(4,3)(4.5,2.857)(5,2.75)(7,2.5)(9,2.375)
   \psline[linewidth=0.5pt]{<->}(2,5.9)(2,2)(10.5,2)
   \psline[linewidth=0.5pt, linecolor=green, linestyle=solid](2,5)(2.5,4)(3,3.5)(3.5,3.2)(4,3)(4.5,2.857)(5,2.75)

   \psline[linewidth=0.5pt, linecolor=blue, linestyle=solid](2.5,4)(2.5,2)
   \psline[linewidth=0.5pt, linecolor=blue, linestyle=solid](2.5,4)(2,4)
   \psline[linewidth=0.5pt, linecolor=blue, linestyle=solid](3,3.5)(3,2)
   \psline[linewidth=0.5pt, linecolor=blue, linestyle=solid](3,3.5)(2.5,3.5)
   \psline[linewidth=0.5pt, linecolor=blue, linestyle=solid](3.5,3.2)(3.5,2)
   \psline[linewidth=0.5pt, linecolor=blue, linestyle=solid](3.5,3.2)(3,3.2)
   \psline[linewidth=0.5pt, linecolor=blue, linestyle=solid](4,3)(4,2)
   \psline[linewidth=0.5pt, linecolor=blue, linestyle=solid](4,3)(3.5,3)
   \psline[linewidth=0.5pt, linecolor=blue, linestyle=solid](4.5,2.857)(4.5,2)
   \psline[linewidth=0.5pt, linecolor=blue, linestyle=solid](4.5,2.857)(4,2.857)
   \psline[linewidth=0.5pt, linecolor=blue, linestyle=solid](5,2.75)(5,2)
   \psline[linewidth=0.5pt, linecolor=blue, linestyle=solid](5,2.75)(4.5,2.75)

   \uput[0](10.1,1.7){$y$}
   \uput[0](1.7,1.7){$y_{0}$}
   \uput[0](4.7,1.7){$y^{*}$}
   \uput[0](1,5){$p(y_{0})$}
 \end{pspicture}
\caption{Lower rectangular and trapezoidal integration rules}
\label{fig:int_rec_trap}
\end{figure}
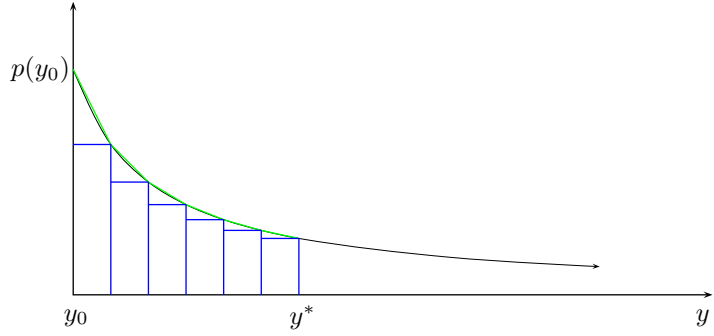

Denote the approximate integrals computed with the lower rectangular rule and the trapezoidal rule as $\sum_{l,h,N}$ and $\sum_{t,h,N}$ respectively, where $N$ is the number of subintervals of length $h$. Formally, these sums are defined as
\[
\sum_{l,h,N}:= \sum_{i=1}^{N}h\cdot p(y_{0}+h\cdot i),
\]
\[
\sum_{t,h,N}:= \sum_{i=1}^{N}\frac{h}{2}\cdot (p(y_{0}+h\cdot i)+p(y_{0}+h\cdot(i-1))).
\]
It should be clear that
\[
 \sum_{t,h,N} = \sum_{l,h,N} + A_{h,N},
\]
where $A_{h,N} = \frac{h}{2}\cdot(p(y_{0})-p(y_{0}+h\cdot N))$.
 
Since function $p(y)$ is concave up due to assumption (D) of the integrating conditions, the integration error on a single subinterval $[y_{0} + h\cdot i, y_{0} + h\cdot(i + 1)]$ is simply an area of one of the regions presented on picture \ref{fig:int_err}, depending on which of the two methods is currently used.

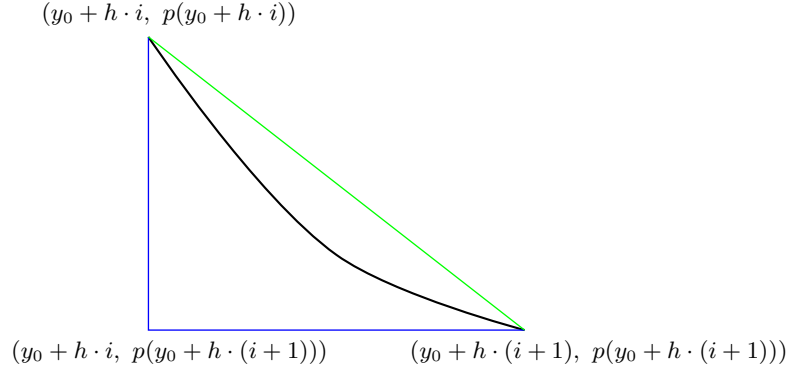
\begin{figure}
\begin{pspicture}(4,1)(8,6.5)
   \pscurve(3,5.9)(5.5,3)(8,2)
   \psline[linewidth=0.5pt, linecolor=blue](3,5.9)(3,2)(8,2)
   \psline[linewidth=0.5pt, linecolor=green, linestyle=solid](3,5.9)(8,2)
   \uput[0](6.3,1.7){\small $(y_{0} + h\cdot(i+1),\mbox{ } p(y_{0} + h\cdot(i+1)))$}
   \uput[0](1,1.7){\small $(y_{0} + h\cdot i, \mbox{ } p(y_{0} + h\cdot(i+1)))$} 
   \uput[0](1.4,6.2){\small $(y_{0} + h\cdot i, \mbox{ } p(y_{0} + h\cdot i))$}
 \end{pspicture}
\caption{Local integration errors for both rules}
\label{fig:int_err}
\end{figure}

More specifically, a region dounded by the blue lines and the black one corresponds to the lower rectangular rule and a region dounded by the green and the black lines corresponds to the trapezoidal rule.

The following relation is obvious, but nevertheless is very important:
\begin{equation}\label{E:inequality}
\sum_{l,h,N} < \int_{y_{0}}^{h\cdot N}p(y)dy < \sum_{l,h,N} + A_{h,N}.
\end{equation}

It allows to formulate the idea of the integrating method as follows.
For a given tolerance $\epsilon$ find such natural numbers $n_{1}$, $n_{2}$, $n_{1}<n_{2}$, that

\begin{itemize}
\item $h\cdot n_{2} - h\cdot n_{1}\leqslant \epsilon$,
\item $\sum_{l,h,n_{1}} + A_{h, n_{1}}\leqslant b \leqslant \sum_{l,h,n_{2}}$.
\end{itemize}
Once such pair is found, according to \eqref{E:inequality} we obtain
\[
 \int_{y_{0}}^{h\cdot n_{1}}p(y)dy < b < \int_{y_{0}}^{h\cdot n_{2}}p(y)dy.
\]
As $\int_{y_{0}}^{z}p(y)dy$ is an increasing function of $z$, it is clear that 
\[
h\cdot n_{1}<y(b)<h\cdot n_{2}. 
\]
Obviously, in either case
\[
 |y(b) - h\cdot n_{1}|<\epsilon,
\]
\[
 |y(b) - h\cdot n_{2}|<\epsilon.
\]
Thus set the approximate solution of \eqref{E:integrating_equation_final} as $y_{b} = h\cdot n_{1}$ or $y_{b} = h\cdot n_{2}$. For even better tolerance $\frac{\epsilon}{2}$, set the approximate solution as $y_{b} = \frac{1}{2}\cdot h\cdot(n_{1}+n_{2})$, which is a mid-point between the previous two. 

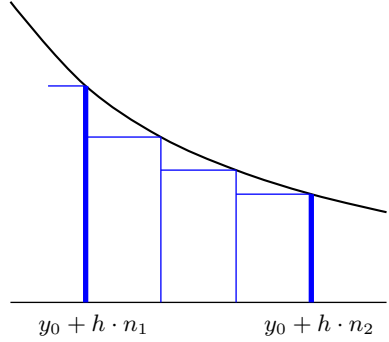
\begin{figure}
\begin{pspicture}(4,1)(8,6.3)
   \pscurve(3,6)(4,4.88)(5,4.2)(6,3.76)(7,3.44)(8,3.2)
   \psline[linewidth=0.5pt, linecolor=black](3,2)(8,2)
   \psline[linewidth=2pt, linecolor=blue, linestyle=solid](4,4.88)(4,2)
   \psline[linewidth=0.5pt, linecolor=blue, linestyle=solid](4,4.88)(3.5,4.88)
   \psline[linewidth=0.5pt, linecolor=blue, linestyle=solid](5,4.2)(5,2)
   \psline[linewidth=0.5pt, linecolor=blue, linestyle=solid](5,4.2)(4,4.2)
   \psline[linewidth=0.5pt, linecolor=blue, linestyle=solid](6,3.76)(6,2)
   \psline[linewidth=0.5pt, linecolor=blue, linestyle=solid](6,3.76)(5,3.76)
   \psline[linewidth=2pt, linecolor=blue, linestyle=solid](7,3.44)(7,2)
   \psline[linewidth=0.5pt, linecolor=blue, linestyle=solid](7,3.44)(6,3.44)
   \uput[0](6.2,1.7){\small $y_{0}+h\cdot n_{2}$}
   \uput[0](3.2,1.7){\small $y_{0}+h\cdot n_{1}$}
%   \uput[0](4.8,0.8){Picture 4}
 \end{pspicture}
\caption{The idea of the integrating method}
\label{fig:idea}
\end{figure}

In other words, we managed to find an approximate solution $y$ of \eqref{E:equation} at point $x=b$ within provided tolerance $\epsilon$, using the integrating method applied to \eqref{E:integrating_equation_final}.
\section{A statement of two algorithms}
Once the idea of the method is clear, it is necessary to formulate an algorithm that realizes this idea consistently. Let superscript $(j)$ correspond to the $j$-th iteration of the algorithm. 

Start with executing the lower rectangular integration of $\int_{y_{0}}^{h\cdot N}p(y)dy$ with constant step $h = h^{(1)} := \epsilon$. Accumulate $\sum_{l,h^{(1)},N}$ until inequality $\sum_{l,h^{(1)},N}\geqslant b$ is met. Denote such $N$ as $N = n_{2}^{(1)}$.
Then go one step back, i.e. to $n_{1}^{(1)} := n_{2}^{(1)} - 1$, and compute $\sum_{t,h^{(1)},n_{2}^{(1)}-1} = \sum_{l,h^{(1)},n_{2}^{(1)}-1} + A_{h^{(1)}, n_{2}^{(1)}-1}$. If $\sum_{t,h^{(1)},n_{2}^{(1)}-1} \leqslant b$, then the approximate solution is found at the first iteration and $y_{b} = \frac{h^{(1)}}{2}\cdot(2\cdot n_{2}^{(1)}-1)$. The algorithm terminates.
However, it is possible that inequality $\sum_{t,h^{(1)},n_{2}^{(1)}-1} \leqslant b$ does not hold. Thus it is necessary to try the second iteration. Set $h^{(2)}:=\frac{h^{(1)}}{2} = \frac{\epsilon}{2}$. Redo the same lower rectangular integration and find $n_{2}^{(2)}$. Once again, it is important to go back in order to check whether the trapezoidal sum is less than or equal to $b$. Only this time, make two backward steps instead of one, from $n_{2}^{(2)}$ to $n_{1}^{(2)} := n_{2}^{(2)}-2$, as we halved $h^{(1)}$ to obtain $h^{(2)}$. Depending on whether inequality $\sum_{t,h^{(2)},n_{2}^{(2)}-2} \leqslant b$ is satisfied or not, we either terminate the algorithm and set the approximate solution as $y_{b} = \frac{h^{(2)}}{2}\cdot(2\cdot n_{2}^{(2)}-2)$ or continue to the next iteration, that is number 3. Note that for $y_{b}$ we are returning a mid-point between the two neighboring ones for better accuracy $\frac{\epsilon}{2}$. It is not mandatory if we simply wish to reach tolerance $\epsilon$, and can choose either $h^{(\cdot)}\cdot n_{1}^{(\cdot)}$ or $h^{(\cdot)}\cdot n_{2}^{(\cdot)}$ for $y_{b}$. 

So the algorithm may be formalized as follows. 
\begin{algorithm}
At iteration $j$, execute:
\begin{enumerate}
 \item $h^{(j)} = \frac{\epsilon}{j}$.
\item Find the smallest such integer $n_{2}$, that $\sum_{l,h^{(j)},n_{2}}\geqslant b$. Denote it as $n_{2}^{(j)}$.
\item If $\sum_{l,h^{(j)},n_{2}^{(j)}-j} + A_{h^{(j)}, n_{2}^{(j)}-j} \leqslant b$, set $y_{b} = \frac{h^{(j)}}{2}\cdot(2\cdot n_{2}^{(j)}-j)$ or $y_{b} = h^{(j)} \cdot(n_{2}^{(j)}-j)$ or $y_{b} = h^{(j)}\cdot n_{2}^{(j)}$. Either one returns an approximation of $y(b)$ within tolerance $\epsilon$. Terminate the algorithm.
\item Else, continue with iteration $j + 1$.
\end{enumerate}
\end{algorithm}
\begin{theorem}\label{T:basic}
 The above algorithm has a finite number of iterations. In other words, it converges.
\end{theorem}
Before we prove this theorem, it is necessary to prove the following lemma first.
\begin{lemma}
\[
 h^{(j)}\cdot n_{2}^{(j)}\leqslant h^{(1)}\cdot n_{2}^{(1)}
\]
for any iteration $j$.
\end{lemma}\label{L:helping}
\begin{proof}
The lower rectangular integration rule applied to $\int p(y)dy$ creates subintervals of length $h^{(1)}=\epsilon$ each on $y$-axis, starting from $y = y_{0}$. The integration at any of the following iterations $j>1$ uses subintervals of length $h^{(j)} = \frac{\epsilon}{j}$ each and thus contains nodes created by the first iteration in its set of nodes. It is also obvious that
\[
 \sum_{i=(k-1)\cdot j+1}^{k\cdot j}p(y_{0}+h^{(j)}\cdot i)\cdot h^{(j)} \geqslant p(y_{0}+h^{(1)}\cdot k)\cdot h^{(1)}
\]
for any natural numbers $k$ and $j$, since $p(y)$ is a decreasing function. Picture \ref{fig:top} demonstrates this property in the case $j=3$. From this, it follows that
\[
 \sum_{l,h^{(j)},j\cdot n_{2}^{(1)}}\geqslant\sum_{l,h^{(1)}, n_{2}^{(1)}}.
\]
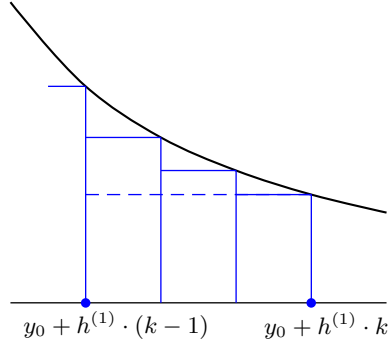
\begin{figure}
\begin{pspicture}(4,1)(8,6.3)
   \pscurve(3,6)(4,4.88)(5,4.2)(6,3.76)(7,3.44)(8,3.2)
   \psline[linewidth=0.5pt, linecolor=black](3,2)(8,2)
   \psline[linewidth=0.5pt, linecolor=blue, linestyle=solid](4,4.88)(4,2)
   \psline[linewidth=0.5pt, linecolor=blue, linestyle=solid](4,4.88)(3.5,4.88)
   \psline[linewidth=0.5pt, linecolor=blue, linestyle=solid](5,4.2)(5,2)
   \psline[linewidth=0.5pt, linecolor=blue, linestyle=solid](5,4.2)(4,4.2)
   \psline[linewidth=0.5pt, linecolor=blue, linestyle=solid](6,3.76)(6,2)
   \psline[linewidth=0.5pt, linecolor=blue, linestyle=solid](6,3.76)(5,3.76)
   \psline[linewidth=0.5pt, linecolor=blue, linestyle=solid](7,3.44)(7,2)
   \psline[linewidth=0.5pt, linecolor=blue, linestyle=solid](7,3.44)(6,3.44)
   \psline[linewidth=0.5pt, linecolor=blue, linestyle=dashed](4,3.44)(7,3.44)
   \psdots[linecolor=blue](4,2)(7,2)
   \uput[0](6.2,1.7){\small $y_{0}+h^{(1)}\cdot k$}
   \uput[0](3,1.7){\small $y_{0}+h^{(1)}\cdot (k-1)$}
 \end{pspicture}
\caption{The lower rectangular integration for $j = 1$ and $j = 3$}
\label{fig:top}
\end{figure}

But
\[
 \sum_{l,h^{(1)}, n_{2}^{(1)}}\geqslant b.
\]
Since, by definition, $n_{2}^{(j)}$ is the smallest such natural number that
\[
 \sum_{l,h^{(j)},n_{2}^{(j)}}\geqslant b,
\]
 it is evident that
\[
 n_{2}^{(j)}\leqslant j\cdot n_{2}^{(1)}.
\]
From here, we obtain 
\[
y_{0} + h^{(j)}\cdot n_{2}^{(j)} \leqslant y_{0} + h^{(j)}\cdot j\cdot n_{2}^{(1)} = y_{0} + h^{(1)}\cdot n_{2}^{(1)}.
\]
This implies the statement of the lemma.
\end{proof}
We are now ready to prove the theorem.

\begin{proof}
Consider sequence $a_{j} = \sum_{t,h^{(j)},n_{2}^{(j)}-j}$, $j>1$. It is evident that
\begin{equation}\label{E:trapezoidal_sum}
 a_{j} = \sum_{l,h^{(j)},n_{2}^{(j)}-j} + A_{h^{(j)},n_{2}^{(j)}-j} = 
\end{equation}
\[
= \sum_{l,h^{(j)},n_{2}^{(j)}-1} + A_{h^{(j)},n_{2}^{(j)}-j} - \sum_{i = n_{2}^{(j)}+1-j}^{n_{2}^{(j)}-1}h^{(j)}\cdot
p(y_{0}+h^{(j)}\cdot i).
\]
Analyze these three terms separately. By definition of $n_{2}^{(j)}$, inequality
\[
 \sum_{l,h^{(j)},n_{2}^{(j)}-1}< b
\]
always holds. Also, according to Lemma 1 and the fact that $p(y)$ is a decreasing function, the second term is bounded in a way presented below:
\[
A_{h^{(j)},n_{2}^{(j)}-j} = \frac{h^{(j)}}{2}\cdot(p(y_{0})-p(y_{0}+h^{(j)}\cdot n_{2}^{(j)}-h^{(1)})) \leqslant
\]
\[
 \leqslant \frac{h^{(j)}}{2}\cdot(p(y_{0})-p(y_{0}+h^{(1)}\cdot n_{2}^{(1)}-h^{(1)})).
\]

The third term is non other than the lower rectangular sum with the step size $h^{(j)}$ between the nodes with numbers $n_{1}^{(j)}$ and $n_{2}^{(j)}-1$, taken with a negative sign. The following lower bound for it will be helpful. Picture \ref{fig:top} helps understand its meaning.
\[
 \sum_{i = n_{2}^{(j)}+1-j}^{n_{2}^{(j)}-1}h^{(j)}\cdot
p(y_{0}+h^{(j)}\cdot i) \geqslant p(y_{0}+h^{(j)}\cdot(n_{2}^{(j)}-1))\cdot(n_{2}^{(j)}-1-n_{1}^{(j)})\cdot h^{(j)}.
\]
As $n_{1}^{(j)} = n_{2}^{(j)}-j$, we obtain
\[
-\sum_{i = n_{2}^{(j)}+1-j}^{n_{2}^{(j)}-1}h^{(j)}\cdot
p(y_{0}+h^{(j)}\cdot i) \leqslant -p(y_{0}+h^{(j)}\cdot(n_{2}^{(j)}-1))\cdot\left(j-1\right)\cdot h^{(j)}.
\]
Apply Lemma 1 and the fact that $p(y)$ decreases:
\[
 p(y_{0}+h^{(j)}\cdot(n_{2}^{(j)}-1))\geqslant p(y_{0}+h^{(1)}\cdot n_{2}^{(1)}).
\]
 So we end up with
\[
 a_{j}<b+\frac{h^{(j)}}{2}\cdot(p(y_{0})-p(y_{0}+h^{(1)}\cdot n_{2}^{(1)}-h^{(1)})) -p(y_{0}+h^{(1)}\cdot n_{2}^{(1)})\cdot(j-1)\cdot h^{(j)}.
\]
The algorithm terminates when inequality $a_{j}\leqslant b$ holds true. It will be then sufficient to require that
\[
 \frac{h^{(j)}}{2}\cdot(p(y_{0})-p(y_{0}+h^{(1)}\cdot n_{2}^{(1)}-h^{(1)})) \leqslant p(y_{0}+h^{(1)}\cdot n_{2}^{(1)})\cdot(j-1)\cdot h^{(j)}.
\]
Isolate $j$ to obtain
\begin{equation}\label{E:criterion}
 j\geqslant 1 + \frac{1}{2}\cdot\left(\frac{p(y_{0})-p(y_{0}+h^{(1)}\cdot n_{2}^{(1)}-h^{(1)})}{p(y_{0}+h^{(1)}\cdot n_{2}^{(1)})}\right).
\end{equation}
Thus, once $j$ becomes large enough to satisfy \eqref{E:criterion}, the trapezoidal sum satisfies inequality $\sum_{t, h^{(j)}, n_{2}^{(j)}-j}\leqslant b$ and the algorithm terminates immediately. 
\end{proof}
\begin{remark}
 For simplicity, condition \eqref{E:criterion} may be replaced by a stronger one:
\begin{equation}\label{E:criterion_simple}
 j\geqslant \frac{1}{2}\cdot\left(1 + \frac{p(y_{0})}{p(y_{0}+h^{(1)}\cdot n_{2}^{(1)})}\right).
\end{equation}
Of course, this may increase the cost.
\end{remark}
Several observations must be pointed out.

 It is easy to see that the iterations in Algorithm 1 are independet in such a sense, that they can be carried out in any order since none of the iterations use the data obtained by the previous ones. They are presented in an order of increasing cost, since at every next iteration there are more nodes to evaluate $p(y)$ at during the lower rectangular integration.

 Denote the smallest integer $j$ satisfying \eqref{E:criterion} as $j_{s}$. The condition that $j\geqslant j_{s}$ is sufficient but is not necessary for the convergence of the method. It may be that $\sum_{t, h^{(j)}, n_{2}^{(j)}-j}$ becomes less than or equal to $b$ before $j$ is large enough to satsify $j\geqslant j_{s}$. The smallest $j$ of all those granting the convergence, denoted $j_{a}$, is not known. This explains why Algorithm 1 starts with smaller $j$ and does not jump right to $j_{s}$ that assures the convergence with the highest cost of all. It seems this strategy may prove cost-saving if the right-hand side of \eqref{E:criterion} is very large and therefore it makes sense to try fewer computations. However, as will be shown below, Algorithm 1 is only valuable from the theoretical point of view and must be avoided in real life computations due to availability of a better algorithm.
\begin{definition}\label{D:trapezoidal}
 $n_{3}^{(j)}$ is the largest such $N$, that satisfies
\[
 \sum_{t,h^{(j)},N}\leqslant b.
\]
\end{definition}
\begin{lemma}\label{L:trapezoidal}
%\begin{equation}\label{E:n2n3}
\[
 n_{2}^{(j)}>n_{3}^{(j)}
\]
%\end{equation}
and
%\begin{equation}\label{E:hn3}
\[
 h^{(j)}\cdot n_{3}^{(j)}\geqslant h^{(1)}\cdot n_{3}^{(1)}.
\]
%\end{equation}
\end{lemma}
\begin{proof}
 By definition of $n_{3}^{(j)}$,
\[
 \sum_{t,h^{(j)},n_{3}^{(j)}}\leqslant b.
\]
But
\[
 \sum_{l,h^{(j)},n_{3}^{(j)}}<\sum_{t,h^{(j)},n_{3}^{(j)}}.
\]
Since
\[
 \sum_{l,h^{(j)},n_{2}^{(j)}}\geqslant b
\]
and $\sum_{l, h^{(j)},N}$ is an increasing function of $N$, the first inequality immediately follows.

To prove the second inequality, we use the idea similar to the one employed in the proof of Lemma 1. By topological properties of a trapezoidal sum applied to a function with a positive second derivative, it is clear that
\[
 \sum_{t,h^{(j)},j\cdot n_{3}^{(1)}}\leqslant\sum_{t,h^{(1)},n_{3}^{(1)}}
\]
for any natural number $j$.
But
\[
 \sum_{t,h^{(1)},n_{3}^{(1)}}\leqslant b.
\]
By definition of $n_{3}^{(j)}$, we get
\[
 n_{3}^{(j)}\geqslant j\cdot n_{3}^{(1)}.
\]
Multiply by $h^{(j)}$ to obtain
\[
 h^{(j)}\cdot n_{3}^{(j)}\geqslant h^{(j)}\cdot j\cdot n_{3}^{(1)} = h^{(1)}\cdot n_{3}^{(1)}.
\]

\end{proof}

\begin{theorem}\label{T:basic_n}
For Algorithm 1 to terminate, it is necessary that $j$ satisfies
\begin{equation}\label{E:criterion_nec}
 j> 1 + \frac{1}{2}\cdot\frac{p(y_{0})-p(y_{0}+h^{(1)}\cdot n_{3}^{(1)}-h^{(1)})-2\cdot p(y_{0}+h^{(1)}\cdot n_{3}^{(1)})}{2\cdot p(y_{0}+h^{(1)}\cdot n_{3}^{(1)}-h^{(1)})}.
\end{equation}
\end{theorem}
\begin{proof}
 Consider again the trapezoidal sum $a_{j}$ in \eqref{E:trapezoidal_sum}. The termination of Algorithm 1 implies that $a_{j}\leqslant b$, or
\[
A_{h^{(j)},n_{2}^{(j)}-j} - \sum_{i = n_{2}^{(j)}+1-j}^{n_{2}^{(j)}-1}h^{(j)}\cdot
p(y_{0}+h^{(j)}\cdot i)\leqslant b-\sum_{l,h^{(j)},n_{2}^{(j)}-1}.
\]
The right-hand side of this inequality is obviously bounded by $h^{(j)}\cdot p(y_{0}+h^{(j)}\cdot n_{2}^{(j)})$, due to the definition of $n_{2}^{(j)}$. Consequently,
\begin{equation}\label{E:inequality_nec}
 A_{h^{(j)},n_{2}^{(j)}-j} - \sum_{i = n_{2}^{(j)}+1-j}^{n_{2}^{(j)}-1}h^{(j)}\cdot
p(y_{0}+h^{(j)}\cdot i)\leqslant h^{(j)}\cdot p(y_{0}+h^{(j)}\cdot n_{2}^{(j)}).
\end{equation}
It is now necessary to find lower bounds for the two left-hand side terms. Interestingly, the proof of Theorem 1 would require to have zero on the right-hand side and the upper bounds of the left-hand side terms at their places. This would lead eventually to \eqref{E:criterion}, whereas we are working on a weaker condition that would apparently lead to \eqref{E:criterion_nec}.

Clearly, due to Lemma 2,
\[
 h^{(j)}\cdot n_{2}^{(j)}>h^{(j)}\cdot n_{3}^{(j)}\geqslant h^{(1)}\cdot n_{3}^{(1)},
\]
so, by decreasing behavior of $p(y)$,
\[
 A_{h^{(j)},n_{2}^{(j)}-j}> \frac{h^{(1)}}{2 j} \cdot(p(y_{0})-p(y_{0}+h^{(1)}\cdot n_{3}^{(1)}-h^{(1)})).
\]
Next, by the same reason,
\[
\sum_{i = n_{2}^{(j)}+1-j}^{n_{2}^{(j)}-1}h^{(j)}\cdot
p(y_{0}+h^{(j)}\cdot i)\leqslant p(y_{0}+h^{(j)}\cdot(n_{2}^{(j)}+1-j))\cdot h^{(j)}\cdot (j-1)\leqslant
\]
\[
 \leqslant p(y_{0}+h^{(1)}\cdot n_{3}^{(1)}-h^{(1)})\cdot h^{(1)}\cdot \left(1-\frac{1}{j}\right)
\]
for $j>1$. Also, the right-hand side of \eqref{E:inequality_nec} is bounded by $\frac{h^{(1)}}{j}\cdot p(y_{0}+h^{(1)}\cdot n_{3}^{(1)})$. So from \eqref{E:inequality_nec} there follows
\[
 \frac{h^{(1)}}{2j}\cdot (p(y_{0})-p(y_{0}+h^{(1)}\cdot n_{3}^{(1)}-h^{(1)}))-p(y_{0}+h^{(1)}\cdot n_{3}^{(1)}-h^{(1)})\cdot h^{(1)}\cdot \left(1-\frac{1}{j}\right)< 
\]
\[
<\frac{h^{(1)}}{j}\cdot p(y_{0}+h^{(1)}\cdot n_{3}^{(1)}).
\]
Solving this inequality for $j$ returns \eqref{E:criterion_nec}.
\end{proof}
\begin{remark}
Condition \eqref{E:criterion_nec} may be replaced by a stronger one:
 \begin{equation}\label{E:criterion_nec_1}
 j> 1 + \frac{1}{2}\cdot\frac{p(y_{0})-3\cdot p(y_{0}+h^{(1)}\cdot n_{3}^{(1)})}{2\cdot p(y_{0}+h^{(1)}\cdot n_{3}^{(1)}-h^{(1)})}.
\end{equation}
\end{remark}

Let the smallest integer $j$ satisfying \eqref{E:criterion_nec} be denoted $j_{n}$. Having now the necessary condition of convergence provided by Theorem 2, it is easy to see that Algorithm 1 is not efficient since it assumes than no information is provided on how fast the integrating method converges. Now that a lower bound for $j_{a}$ is known due to \eqref{E:criterion_nec}, it will be wiser to focus on searching $j_{a}$ only among those $j$ satisfying $j_{n}\leqslant j\leqslant j_{s}$. However, even knowing $j_{n}$ does not provide any better strategy of solving \eqref{E:integrating_equation_final} than simply picking $j:=j_{s}$ immediately and doing the lower rectangular integration once ( not counting the one made with initial step $h^{(1)}=\epsilon$ necessary for computing $j_{s}$ via \eqref{E:criterion} ) with the maximum cost. This may be explained by the fact that the search for $j_{a}$ itself may result in a total cost higher than that of a single integration with the smallest step $h^{(j_{s})}$.  

Suppose, for instance, the bisection method is used to find $j_{a}$. At the first try, $j_{1} \simeq \frac{1}{2}\cdot(j_{n}+j_{s})$. Assuming that no information other than the values of $j_{n}$ and $j_{s}$ is provided, the probability that $j_{1}$ will grant convergence is about 50 \%. The computational cost $C_{bisection,1}$ in this case is estimated roughly via the formula
\[
 C_{bisection,1} \approx \frac{j_{n}+j_{s}}{2}\cdot \frac{b}{\epsilon} + \frac{b}{\epsilon}.
\]
The last term comes from the very first integration with step $h^{(1)} = \epsilon$. It is necessary in order to obtain $j_{n}$ and $j_{s}$. At the same time, the computational cost $C_{real}$ that corresponds to the case $j = j_{s}$ is estimated via
\[
 C_{real} \approx j_{s}\cdot\frac{b}{\epsilon} + \frac{b}{\epsilon}.
\]
Obviously, $\frac{C_{bisection,1}}{C_{real}} > 0.5$, which means we do not even drop to the half of the maximum computational cost when we use $j = j_{1}$. The probability of 50 \% is not high enough to convince us to try $j_{1}$. Even worse, already the second iteration $j_{2} = \frac{1}{2}\cdot(j_{1}+j_{s})$ with 75 \% of success probability requires more computational time than the case with $j_{s}$ since
\[
 C_{bisection,2}\approx \frac{\frac{j_{n}+j_{s}}{2}+j_{s}}{2}\cdot \frac{b}{\epsilon} + C_{bisection,1} > C_{real}.
\]
So it is evident that the algorithm which uses $h^{(j_{s})}$ as soon as $j_{s}$ is provided, is the best choice from the point of view of computational efforts. It is formulated as follows.

\begin{algorithm}
Perform the following two actions.
\begin{enumerate}
 \item Execute the lower rectangular integration with step $h^{(1)}=\epsilon$ until inequality $\sum_{l,h^{(1)},n_{2}^{(1)}}\geqslant b$ is met. Check if $\sum_{t, h^{(1)}, n_{2}^{(1)}-1}\leqslant b$. If it is true, terminate the algorithm having set $y_{b}=h^{(1)}\cdot n_{2}^{(1)}$.
Else, continue to step 2.
 \item Find the smallest such integer $j$ that satisfies \eqref{E:criterion} or \eqref{E:criterion_simple}. Compute $h^{(j)} = \frac{\epsilon}{j}$. Execute the lower rectangular integration the second and the last time, with step $h^{(j)}$, setting $y_{b}=h^{(j)}\cdot n_{2}^{(j)}$. Terminate the algorithm. 
\end{enumerate}
\end{algorithm}
%As an additional piece of advice for numerical computing, it is reasonable to execute the integration process initially with the doubled tolerance, i.e. from $h^{(1)} = 2\epsilon$, aiming for the mid-point $y_{b}$ between $h^{(j)}\cdot n_{1}^{(j)}$ and $h^{(j)}\cdot n_{2}^{(j)}$ in the end of the algorithm, thus still reaching the desired accuracy determined by $\epsilon$.

\section{Advantages and disadvantages}
The main advantage of the integrating method (and apparently its right for existence) is that it returns solution $y(b)$ of \eqref{E:equation} within provided tolerance $\epsilon$. This is what makes it unique and distinguished from the widely used Runge-Kutta or Adams methods that return the answer with an error in the form $O(h^{p})$.  

The other advantage is a stability of the method. Since it uses integration process, the method is always stable, regardless of given $\epsilon$ and integration step $h$. This is not always true for other solvers of ordinary differential equations. For example, forward Euler method may experience unstability thus becoming unusable for large time-steps.

Another obvious advantage is a simplicity of the algorithm. It is easy to understand and implement in any mathematical software.

One of the disadvantages of the integrating method is that it is not general and only works for scalar problems having form $\eqref{E:equation_main}$ and satifying the integrating conditions (A)-(D) imposed on functions $f(y)$ and $g(x)$. It is also important to know the value of $\int_{y_{0}}^{+\infty}p(y)dy$ ( or at least some approximation of it ) in case this integral converges. It may be challenging to use the integrating method without being sure whether the problem is solvable in general or solvable within a reasonable amount of time. The latter may be an issue if $b$ is very close to $c$ which is the right limit of the solution extension interval $[0, c)$.  A relatively high cost of the method comes from the integrating process with small step size and is considered another disadvantage.
\begin{remark}
All these advantages and disadvantages remind of similar situation with numerical methods for solving nonlinear algebraic equations. The bisection method is known for its ability to return the solution of $f(x)=0$ on interval $[a,b]$ within a given tolerance $\epsilon$ as long as $f(a)\cdot f(b)<0$. The method, however, cannot be extended to a class of vector equations. It also does not converge as fast as the Newton's method. The latter can also be used for vector equations. But in return, it cannot assuredly present the answer within the given tolerance $\epsilon$.
\end{remark}
Another property of the integrating method that may seem as a disadvantage is that it works towards finding $y(b)$ only, whereas all known step solvers return intermidiate points as well. They therefore present an approximation of the whole solution curve $y(x)$ on the presented mesh $x_{k}$, $0< x_{k}\leqslant b$. The fact that the integrating method manages to find $y(b)$ within tolerance $\epsilon$ does not necessarily imply that intermidiate points $y(x_{k})$ are approximated within $\epsilon$ too. In order to compute these intermidiate points, the integrating method may be applied separately to each point. This will result in a very high total cost.

The last problem may be avoided with help of inequality $\eqref{E:criterion}$.

\begin{theorem}
 Let $\{0, x_{1}, x_{2}, ... , x_{N-1}, x_{N}\}$ with $x_{N} = b$ be the mesh provided by a user and $n_{2}^{(1)}$ is the smallest such integer that $\sum_{l,h^{(1)}, n_{2}^{(1)}}\geqslant b$ with $h^{(1)}=\epsilon$. Let $j$ be an integer that satisfies \eqref{E:criterion} or \eqref{E:criterion_simple}, and the lower rectangular integration is executed once with step $h^{(j)} = \frac{\epsilon}{j}$. For every $x_{i}$ set $y_{x_{k}}:= h^{(j)}\cdot n_{2,k}^{(j)}$, where $n_{2,k}^{(j)}$ is the smallest integer such that $\sum_{l,h^{(j)},n_{2,k}^{(j)}}\geqslant x_{k}$. Then the following holds true:
\[
 |y(x_{k})-y_{x_{k}}|<\epsilon\mbox{ }\forall k,\mbox{ } 1\leqslant k\leqslant N.
\]
\end{theorem}
The above theorem assures that if $j$ satisfies \eqref{E:criterion} or \eqref{E:criterion_simple}, concerned only with the final node $x_{N}=b$, the solution at the intermidiate mesh nodes is already to be evaluated within tolerance $\epsilon$ during the lower rectangular integration aiming to find $y_{b}$. This is a very relieving result. 
\begin{proof}
 Pick any intermidiate point $x_{k}$, $0<x_{k}<b$. When the integrating method is used specifically to find $y_{x_{k}}$, then it will be sufficient to use the lower rectangular integration with step size $h^{(j_{k})} = \frac{\epsilon}{j_{k}}$, where integer $j_{k}$ is the smallest integer satisfying
\[
 j_{k}\geqslant 1 + \frac{1}{2}\cdot\left(\frac{p(y_{0})-p(y_{0}+\epsilon\cdot n_{2,k}^{(1)}-\epsilon)}{p(y_{0}+\epsilon\cdot n_{2,k}^{(1)})}\right),
\]
with $n_{2,k}^{(1)}$ being the smallest such integer that satisfies $\sum_{l,\epsilon, n_{2,k}^{(1)}}\geqslant x_{k}$. Since $x_{k}<b$, it is obvious that $n_{2,k}^{(1)}\leqslant n_{2,N}^{(1)}=n_{2}^{(1)}$. So
\[
 p(y_{0}) - p(y_{0}+\epsilon\cdot n_{2}^{(1)}-\epsilon)\geqslant p(y_{0})-p(y_{0}+\epsilon\cdot n_{2,k}^{(1)}-\epsilon).
\]
By the same reason,
\[
 \frac{p(y_{0}) - p(y_{0}+\epsilon\cdot n_{2}^{(1)}-\epsilon)}{p(y_{0}+\epsilon\cdot n_{2}^{(1)})}\geqslant \frac{p(y_{0})-p(y_{0}+\epsilon\cdot n_{2,k}^{(1)}-\epsilon)}{p(y_{0}+\epsilon\cdot n_{2,k}^{(1)})}.
\]
So we see that if $j$ satisfies \eqref{E:criterion}, it automatically satisfies
\[
 j\geqslant 1 + \frac{1}{2}\cdot\left(\frac{p(y_{0})-p(y_{0}+\epsilon\cdot n_{2,k}^{(1)}-\epsilon)}{p(y_{0}+\epsilon\cdot n_{2,k}^{(1)})}\right).
\]
This proves that $j\geqslant j_{k}$. This assures the integration step $h^{(j)}=\frac{\epsilon}{j}$ is small enough to guarantee that $|y(x_{k})-y_{x_{k}}|< \epsilon$ for arbitrary $x_{k}$, $0<x_{k}< b$.
\end{proof}
\begin{remark}
Similar proof may be conducted in case of inequality \eqref{E:criterion_simple}.
\end{remark}
\begin{remark}
Similar result exists for the necessary condition \eqref{E:criterion_nec}. It simply states that the larger integer $k$ is, the stronger condition \eqref{E:criterion_nec} for point $x_{k}$ is. In other words, if \eqref{E:criterion_nec} is satisfied for the final point $b$, it is automatically satisfied for the previous points $x_{k}$ of the mesh. The same may be said in case of inequality \eqref{E:criterion_nec_1}. This result seem to present no value for the current work though.
\end{remark}
Now that Theorem 2 is proven, Algorithm 2 may be assuredly used to find the approximate solution of \eqref{E:equation} on arbitrary mesh $\{0, x_{1}, ..., x_{N-1}, x_{N}\}$, $x_{N} = b$, with almost the same cost as for solely the final point $b$.
 
\begin{conjecture}
 The integrating method may be improved to cover a wider area of initial-value problems than just those restricted by the integrating conditions (A)-(D). If it is possible to track all the inflection points of function $\frac{1}{f(y)}$ from \eqref{E:equation_main} precisely, it may be possible to use a lower rectangular integration on intervals where 
\[
\left(\frac{1}{f(y)}\right)^{''} > 0
\]
and a higher rectangular integration on intervals where
\[
\left(\frac{1}{f(y)}\right)^{''} < 0.
\]
Using relations analogous to \eqref{E:inequality} may help construct algorithms to find solutions within tolerance $\epsilon$ even for this more general class of problems. This may be a topic of further research.
\end{conjecture}

\section{Numerical experiments}
Numerical experiments were conducted to test both Algorithm 1 and 2 on two initial-value problems:
\begin{equation}\label{E:test1}
\begin{cases}
 \frac{dy}{dx} = y + 1,\\
 y(0) = 0,
\end{cases}
\end{equation}
with exact solution $y(x) = e^{x}-1$ and mesh $x_{k} = 0.05\cdot k$, $k\leqslant20$, and
\begin{equation}\label{E:test2}
\begin{cases}
 \frac{dy}{dx} = y^{2},\\
 y(0) = 0.5,
\end{cases}
\end{equation}
with exact solution $y = \frac{1}{2-x}$ and mesh $x_{k} = 0.05\cdot k$, $k\leqslant32$. It is important to note that in the second case the solution may only be extended up to $x = 2$ since the integral
\[
 \int_{0.5}^{+\infty}\frac{dy}{y^{2}}
\]
converges and is equal to $2$. In the first case,
\[
 \int_{0}^{+\infty}\frac{dy}{y+1} = +\infty,
\]
so no problems are encountered. 

Tolerance was $\epsilon = 10^{-4}$.
Both algorithms were implemented in the mathematical software Octave. Algorithm 1 was applied separately to each node $x_{k}$. The computational time was measured on a computer with Intel Core i7 processor with 2.80 GHz frequency. The computation started with doubled tolerance $h^{(1)} = 2\epsilon$ and the approximate solution $y_{x_{k}}$ was found as a mid-point between $y_{0}+h^{(j)}\cdot n_{1}^{(j)}$ and $y_{0}+h^{(j)}\cdot n_{2}^{(j)}$ in both cases. Tables 1 and 2 present results of Algorithm 1 applied to problems \eqref{E:test1} and \eqref{E:test2}. The first case required 4.3 seconds, and the second one required 75.6 seconds. 

\begin{center}
\begin{tabular}{||l|c|c|c|c|c||}
 \hline $x_{k}$ & $y_{x_{k}}$ & $|y(x_{k})-y_{x_{k}}|\cdot 10^{4}$ & The actual number of iterations & $j_{n}$ & $j_{s}$\\
 \hline 0.05 & 0.0513 & 0.289 & 1 & 1 & 2 \\
 \hline 0.1 & 0.1051 & 0.709 & 1 & 1 & 2 \\
 \hline 0.15 & 0.1619 & 0.658 & 1 &1 & 2 \\
 \hline 0.2 & 0.2215 & 0.972 & 1 &1 & 2 \\
 \hline 0.25 & 0.2841 & 0.746 & 1 &1 & 2 \\
 \hline 0.3 & 0.3499 & 0.412 & 1 & 1 &2 \\
 \hline 0.35 & 0.4191 & 0.326 & 1 & 1 &2 \\
 \hline 0.4 & 0.4919 & 0.753 & 1 & 1 &2 \\
 \hline 0.45 & 0.5683 & 0.122 & 1 & 1 &2 \\
 \hline 0.5 & 0.6487 & 0.213 & 1 & 1 &2 \\
 \hline 0.55 & 0.7333 & 0.4698 & 1 & 1 &2 \\
 \hline 0.6 & 0.8221 & 0.188 & 2 & 1 &2 \\
 \hline 0.65 & 0.9155 & 0.408 & 2 & 1 &2 \\
 \hline 0.7 & 1.0138 & 0.473 & 2 & 1 &2 \\
 \hline 0.75 & 1.1171 & 0.9998 & 1 & 1 &2 \\
 \hline 0.8 & 1.2256 & 0.591 & 2 & 1 &2 \\
 \hline 0.85 & 1.3397 & 0.531 & 1 & 1 &2 \\
 \hline 0.9 & 1.4597 & 0.969 & 1 & 1 &2 \\
 \hline 0.95 & 1.5857 & 0.097 & 2 & 1 &2 \\
 \hline 1.0 & 1.7183 & 0.182 & 2 & 1 &2 \\
 \hline
\end{tabular}
\end{center}
\begin{center}
 Table 1: results of Algorithm 1 applied to each $x_{k}$ for solving \eqref{E:test1}
\end{center}

Results of Algorithm 2 applied to \eqref{E:test1} and \eqref{E:test2} are presented in Tables 3 and 4 respectively. The initial step $h^{(1)} = 2\epsilon$ was divided by 2 in case \eqref{E:test1} and by 14 in case \eqref{E:test2}. This agrees with the cell value of the last line and the last column of Tables 1 and 2 respectively. Algorithm 2 significantly speeds up computations since it only took 0.96 seconds to obtain the results of Table 3 and 6.04 seconds for Table 4. This may be explained by the fact that, although inequality \eqref{E:criterion} is not a criterion of the algorithm termination, it returns a value close to the actual minimum number $j$ by which $h^{(1)}$ must be divided in order to grant required tolerance at every fixed node $x_{k}$. Both Tables 1 and 2 show that the difference between $j_{s}$ and $j_{s}$ is no larger than 1 for all nodes. If the fact that $j_{n}$ and $j_{s}$ are close to each other remains true in many other cases, it is another plus towards Algorithm 2. The author recommends using it in real life applications where a need to reach the desired accuracy with 100 \% guarantee is higher than a need for computational speed and efficiency.

\pagebreak

\hspace{10pt}
\begin{center}
\begin{tabular}{||l|c|c|c|c|c||}
 \hline $x_{k}$ & $y_{x_{k}}$ & $|y(x_{k})-y_{x_{k}}|\cdot 10^{4}$ & The actual number of iterations & $j_{n}$ & $j_{s}$\\
 \hline 0.05 & 0.5129 & 0.795 & 1 & 1 & 2 \\
 \hline 0.1 & 0.5263 & 0.158 & 1 & 1 &2 \\
 \hline 0.15 & 0.5405 & 0.405 & 1 & 1 &2 \\
 \hline 0.2 & 0.5555 & 0.556 & 1 & 1 &2 \\
 \hline 0.25 & 0.5715 & 0.714 & 1 & 1 &2 \\
 \hline 0.3 & 0.5883 & 0.647 & 1 & 1 &2 \\
 \hline 0.35 & 0.6061 & 0.394 & 1 & 1 &2 \\
 \hline 0.4 & 0.6250 & 0.000 & 2 & 1 &2 \\
 \hline 0.45 & 0.6451 & 0.613 & 2 & 1 &2 \\
 \hline 0.5 & 0.6667 & 0.333 & 1 & 1 &2 \\
 \hline 0.55 & 0.6897 & 0.448 & 1 & 1 &2 \\
 \hline 0.6 & 0.7143 & 0.143 & 1 & 1 &2 \\
 \hline 0.65 & 0.7408 & 0.593 & 2 & 1 &2 \\
 \hline 0.7 & 0.7693 & 0.692 & 1 & 1 &2 \\
 \hline 0.75 & 0.8000 & 0.000 & 2 & 1 &2 \\
 \hline 0.8 & 0.8334 & 0.667 & 2 & 1 &2 \\
 \hline 0.85 & 0.8696 & 0.348 & 2 & 2 &3 \\
 \hline 0.9 & 0.9091 & 0.091 & 3 & 2 &3 \\
 \hline 0.95 & 0.9524 & 0.524 & 3 &2 & 3 \\
 \hline 1.0 & 1.0000 & 0.333 & 3 & 2 &3 \\
 \hline 1.05 & 1.0527 & 0.684 & 3 & 2 &3 \\
 \hline 1.1 & 1.1112 & 0.556 & 3 & 2 &3 \\
 \hline 1.15 & 1.1766 & 0.961 & 3 & 3 &4 \\
 \hline 1.2 & 1.2501 & 0.500 & 4 & 3 &4 \\
 \hline 1.25 & 1.3334 & 0.667 & 4 & 4 &5 \\
 \hline 1.3 & 1.4286 & 0.486 & 5 & 4 &5 \\
 \hline 1.35 & 1.5385 & 0.785 & 5 & 5 &6 \\
 \hline 1.4 & 1.6667 & 0.619 & 7 & 6 &7 \\
 \hline 1.45 & 1.8183 & 0.896 & 7 & 7 &8 \\
 \hline 1.5 & 2.0001 & 0.778 & 9 & 8 &9 \\
 \hline 1.55 & 2.2223 & 0.978 & 10 & 10 &11 \\
 \hline 1.6 & 2.5001 & 0.857 & 14 & 13 &14 \\
 \hline
\end{tabular}
\end{center}
\begin{center}
 Table 2: results of Algorithm 1 applied to each $x_{k}$ for solving \eqref{E:test2}
\end{center}

\pagebreak

\hspace{10pt}
\begin{center}
\begin{tabular}{||l|c|c||}
 \hline $x_{k}$ & $y_{x_{k}}$ & $|y(x_{k})-y_{x_{k}}|\cdot 10^{4}$ \\
 \hline 0.05 & 0.05120 & 0.711 \\
 \hline 0.1 & 0.1051 & 0.709 \\
 \hline 0.15 & 0.1618 & 0.342 \\
 \hline 0.2 & 0.2214 & 0.028 \\
 \hline 0.25 & 0.2840 & 0.254 \\
 \hline 0.3 & 0.3498 & 0.588 \\
 \hline 0.35 & 0.4190 & 0.675 \\
 \hline 0.4 & 0.4918 & 0.247 \\
 \hline 0.45 & 0.5683 & 0.122 \\
 \hline 0.5 & 0.6487 & 0.213 \\
 \hline 0.55 & 0.7332 & 0.530 \\
 \hline 0.6 & 0.8221 & 0.188 \\
 \hline 0.65 & 0.9155 & 0.408 \\
 \hline 0.7 & 1.0138 & 0.473 \\
 \hline 0.75 & 1.1170 & 0.0002 \\
 \hline 0.8 & 1.2256 & 0.591 \\
 \hline 0.85 & 1.3397 & 0.531 \\
 \hline 0.9 & 1.4596 & 0.031 \\
 \hline 0.95 & 1.5857 & 0.097 \\
 \hline 1.0 & 1.7183 & 0.182 \\
 \hline
\end{tabular}
\end{center}
\begin{center}
 Table 3: results of Algorithm 2 applied to \eqref{E:test1}
\end{center}

\begin{center}
\begin{tabular}{||l|c|c||}
 \hline $x_{k}$ & $y_{x_{k}}$ & $|y(x_{k})-y_{x_{k}}|\cdot 10^{4}$\\
 \hline 0.05 & 0.5127 & 0.919 \\
 \hline 0.1 & 0.5262 & 0.872 \\
 \hline 0.15 & 0.5404 & 0.977 \\
 \hline 0.2 & 0.5555 & 0.841 \\
 \hline 0.25 & 0.5713 & 0.8571\\
 \hline 0.3 & 0.5881 & 0.924  \\
 \hline 0.35 & 0.6060 & 0.892 \\
 \hline 0.4 & 0.6249 & 0.857 \\
 \hline 0.45 & 0.6451 & 0.899\\
 \hline 0.5 & 0.6666 & 0.810  \\
 \hline 0.55 & 0.6896 & 0.837\\
 \hline 0.6 & 0.7142 & 0.857  \\
 \hline 0.65 & 0.7407 & 0.836 \\
 \hline 0.7 & 0.7691 & 0.879 \\
 \hline 0.75 & 0.7999 & 0.857 \\
 \hline 0.8 & 0.8333 & 0.762  \\
 \hline 0.85 & 0.8695 & 0.795 \\
 \hline 0.9 & 0.9090 & 0.766 \\
 \hline 0.95 & 0.9523 & 0.810 \\
 \hline 1.0 & 0.9999 & 0.714  \\
 \hline 1.05 & 1.0526 & 0.744 \\
 \hline 1.1 & 1.1110 & 0.683  \\
 \hline 1.15 & 1.1764 & 0.563 \\
 \hline 1.2 & 1.2499 & 0.571  \\
 \hline 1.25 & 1.3333 & 0.476 \\
 \hline 1.3 & 1.4285 & 0.429  \\
 \hline 1.35 & 1.5384 & 0.330 \\
 \hline 1.4 & 1.6666 & 0.238 \\
 \hline 1.45 & 1.8182 & 0.104 \\
 \hline 1.5 & 2.0000 & 0.143  \\
 \hline 1.55 & 2.2223 & 0.349 \\
 \hline 1.6 & 2.5001 & 0.857 \\
 \hline
\end{tabular}
\end{center}
\hspace{2pt}
\begin{center}
 Table 4: results of Algorithm 2 applied to \eqref{E:test2}
\end{center}

\renewcommand\refname{Related reading}

\end{document}